\documentclass[11pt,a4paper]{article}
\usepackage{amsmath,amssymb,amsthm}
\usepackage{enumerate}
\usepackage{hyperref}
\usepackage{mathrsfs}
\usepackage[a4paper]{geometry}
\hypersetup{
 pdftitle={Rank-width and Tree-width  of  $H$-minor-free Graphs},
 pdfauthor={Fedor V. Fomin, Sang-il Oum, Dimitrios M. Thilikos},
 pdfkeywords={rank-width, tree-width, clique-width}
}
\newtheorem{THM}{Theorem}
\newtheorem{LEM}[THM]{Lemma}

\newtheorem{PROP}[THM]{Proposition}
\theoremstyle{definition}

\newcommand\cwd{\operatorname{\bf cw}}

\newcommand\rwd{\operatorname{\bf rw}}
\renewcommand\deg{\operatorname{\bf deg}}
\newcommand\twd{\operatorname{\bf tw}}
\newcommand\cutrk{\rho}
\newcommand\rank{\operatorname{\bf rank}}

\newcommand\lab{\operatorname{lab}}

\newcommand\B{\mbox{\boldmath $B$}}
\newcommand\abs[1]{\lvert#1\rvert}
\newcommand\genus{\varepsilon}
\begin{document}
\title{Rank-width and Tree-width  of $H$-minor-free Graphs}
\author{
 Fedor V. Fomin\thanks{{\tt fomin@ii.uib.no}} \footnote{
Supported by the Norwegian Research 
Council.
 }\\
  Department of Informatics\\
  University of Bergen,\\
  N-5020 Bergen, Norway.
 \and
 Sang-il Oum\thanks{{\tt sangil@kaist.edu}} \footnote{
   Supported by  Basic Science Research Program
   through the National Research Foundation of Korea (NRF)
   funded by the Ministry of Education, Science and Technology (2009-0063183).
 }\\
 Department of Mathematical Sciences \\
 KAIST\\
 Daejeon, 305-701,
 Republic of Korea.
 \and
 Dimitrios M. Thilikos\thanks{{\tt sedthilk@math.uoa.gr}} \footnote{Supported by the project  ``Kapodistrias'' (A$\Pi$ 02839/28.07.2008) of the National and Kapodistrian University of Athens (project code: 70/4/8757).   
 }\\
  Department of Mathematics\\
  National and Kapodistrian University of Athens\\
  Panepistimioupolis, GR-157 84, Athens, Greece.
}

% \date\today
\date{September 18, 2009}
\maketitle

\begin{abstract}
  \noindent
  We prove that for any fixed $r\geq 2$, the tree-width of graphs not
  containing $K_{r}$ as a topological minor (resp. as a subgraph) is
  bounded by a linear (resp. polynomial)
  function of their  rank-width. We also present refinements of our bounds for
 other graph classes such as $K_{r}$-minor free graphs and graphs of bounded genus.
\end{abstract}

\section{Introduction}\label{sec:intro}

Tree-width and rank-width are width parameters of
graphs, which are, roughly speaking,
measures of their decomposability. These parameters play very important roles in Structural and Algorithmic Graph Theory. 
For example, if we restrict the input to graphs of bounded  tree-width
or rank-width,
then many problems that are NP-hard in general can be solved in
polynomial time.
%because those graphs admit tree-like
%    decompositions which  allow  dynamic programming approach.

It is natural to ask about the relations between various width
parameters of graphs. Let us write $\twd(G)$ and $\rwd(G)$
for tree-width and rank-width of a graph $G$, respectively.
As it was shown by Oum~\cite{Oum2006c},  for any graph $G$ 
%
%Graphs of bounded tree-width have bounded rank-width
%by the inequality
\begin{equation}
 \rwd(G)\le \twd(G)+1.
 \label{eq:rwdtwd}
\end{equation}
%by Oum~\cite{Oum2006c}.
On the other hand,  there is no function~$f$ such that $\twd(G)\le f(\rwd(G))$.
For instance, the complete graph $K_n$ on $n$ vertices has  tree-width $n-1$
and  rank-width $1$. 
However, the situation changes when   we impose some conditions on the structure of graph $G$.
%, then we may have some
%function~$f$ such that $\twd(G)\le f(\rwd(G))$.
 Courcelle and Olariu~\cite{CO2000} proved that such functions $f$
exist under various conditions. Actually, their paper
is about the clique-width of
graphs, which has been defined earlier than the rank-width. In
fact,  the rank-width was defined by Oum and Seymour~\cite{OS2004}
so that graphs have bounded rank-width if and only if they have
bounded clique-width. More precisely, they
proved that
\begin{equation}
 \label{eq:rwdcw}
 \rwd(G)\le \cwd(G)\le 2^{\rwd(G)+1}-1,
\end{equation}
where $\cwd(G)$ denotes the clique-width of a graph $G$.

%The first bound for clique-width in terms of tree-width appeared for the class of graphs
%excluding $K_{r,r}$ as a subgraph by 
In particular, Courcelle and Olariu {\cite[Theorem 5.9]{CO2000}} have shown that 
   for every positive integer $r$, there exists a function~$f_r$ such that 
if a graph $G$ has no subgraph isomorphic to the complete bipartite graph $K_{r,r}$ on $2r$ vertices, then $\twd(G) \le f_{r}(\cwd(G))$.
The proof by Courcelle and Olariu is highly non-constructive.
Later, Gurski and Wanke \cite{GW2000} proved that
if a graph $G$ has no subgraph isomorphic to $K_{r,r}$, then
\begin{equation}\label{eq:gw}
\twd(G)+1 \le 3(r-1) \cwd(G).
\end{equation}
%\begin{THM}[Gurski and Wanke~\cite{GW2000}]
%\label{eq:gw}
%If a graph $G$ has no subgraph isomorphic to $K_{r,r}$, then $\twd(G) \le 3\cdot(r-1)\cdot \cwd(G)-1$.
%\end{THM}
By combining (\ref{eq:gw})  with (\ref{eq:rwdcw}),
we can directly deduce that for
every graph   $G$ having no $K_{r,r}$ as a subgraph,
\begin{equation}
 \label{eq:rwtw}
\twd(G)+1\le 3(r-1)(2^{\rwd(G)+1}-1).
\end{equation}
In this paper, we  show that the  exponential bound \eqref{eq:rwtw}
%on the tree-width  with respect to the rank-width
can be
improved to a \emph{polynomial} bound  for graphs not
containing $K_{r,r}$
as a subgraph and to a \emph{linear} bound for graphs not containing
$K_{r}$ as a  minor or a topological minor. We will apply  our proof techniques
to various classes of graphs while still obtaining linear bounds.
%the case when $G$ provide improvements of
%this bound for graphs of bounded genus.

\medskip 

Let us summarize our theorems as follows. The results are ordered with respect to the generality of the corresponding class. 
In what follows 
%\begin{THM}
%\label{thm:ttttt}
 $G$ is  a graph with at least one edge.  We refer to  Section~\ref{sec:def} for the definitions of graph classes. 

\begin{itemize}
\item Theorem~\ref{thm:surface}: If  $G$ is  planar, then
  \begin{align*}
  %\rwd(G)\le \cwd(G)&<  12 \rwd(G),\\
  %\rwd(G)\le
  \twd(G)&< 72\rwd(G) -1.
  \end{align*}
   %(Theorem~\ref{thm:surface}) 
\item Theorem~\ref{thm:surface}: If the Euler genus of $G$ is at most $g$, then
  \begin{align*}
    %\rwd(G)\le \cwd(G)&< 12\rwd(G)+10r,\\
     \twd(G)&< 3(2+\sqrt{2g})(6\rwd(G)+5g)-1.
 \end{align*}
\item Theorem~\ref{thm:minor}: %There is a fixed constant $\mu$ such that
  If  $G$ contains  no $K_r$ as a minor,  $r>2$, then
  \begin{align*}
    %\rwd(G)\le \cwd(G)&<2 \cdot 2^{\mu r\log\log r}\rwd(G),\\
    \twd(G)&= 2^{O( r\log\log r)}\rwd(G)
 \end{align*}
  
\item Theorem~\ref{thm:topminor}: %There is a fixed constant $\tau$ such that
  If  $G$ contains  no $K_{r}$ as a topological minor
  for $r>2$, then
  \begin{align*}
  % \rwd(G)\le \cwd(G)&< 2\cdot 2^{\tau r\log r} \rwd(G),\\
    \twd(G)=   2^{O( r\log r)} \rwd(G).
  \end{align*}
  
  \item Theorem~\ref{thm:topgrad}: 
  If  $\nabla_{1}(G)\leq r$, %$G$ is a graph with the greatest reduced average degree  with rank $1$,  $\nabla_{1}(G)\leq r$
  then
  \begin{align*}
  %\rwd(G)\le    \cwd(G)&< 2\cdot 4^{r} \rwd(G),\\
       \twd(G) &< 12\cdot  r \cdot  4^{r} \rwd(G) -1.
 \end{align*} Here,  $\nabla_{1}$ is the greatest reduced average degree  with rank $1$. 
 \item Theorem~\ref{thm:krr}: If $G$ has no  subgraph isomorphic to $K_{r,r}$ for $r\ge 2$, then
  \begin{align*}
   % \rwd(G)\le \cwd(G)&< \frac{2(r-2)}{r+1} \binom{\rwd(G)}{r}+
  %  2\sum_{i=0}^r \binom{\rwd(G)}{i} ,\\
     \twd(G)&<3(r-1)\left(\frac{2(r-2)}{r+1} \binom{\rwd(G)}{r}+
    2\sum_{i=0}^r \binom{\rwd(G)}{i}\right) -1.
 \end{align*}
\end{itemize}
%\end{THM}

\section{Definitions}\label{sec:def}
In this paper all graphs  are simple undirected graphs without
loops and parallel edges.
%For a graph $G=(V,E)$ we denote by $V(G)$ and $E(G)$ the sets of vertices and
%edges of $G$ respectively.
For a vertex $v\in V(G)$ of graph $G$,
we denote by $N_{G}(v)$ the set of vertices  in $G$ that are adjacent
to $v$ and
we write $\deg_{G}(v)=|N_{G}(v)|$ to denote the \emph{degree} of a vertex $v$ in $G$.
%For a set $X\subseteq V(G)$ we define its neighborhood
%\[
%N(X)=(\cup_{v\in X} N(v))\setminus X.
%\]
%For graphs $G$ and $H$,  we define
The \emph{union} $G\cup H$
of two graphs $G$ and $H$
is  the graph such that
$V(G\cup H)=V(G)\cup V(H)$, and $E(G\cup H)= E(G)\cup E(H)$.
Two distinct vertices $x,y$ of $G$ are
\emph{twins} if there are no vertices in $V(G)\setminus \{x,y\}$
that are adjacent to exactly one of $x$ and $y$, or equivalently 
 $N_{G}(x)\setminus\{x,y\}=N_{G}(y)\setminus\{x,y\}$.
A \emph{clique} of a graph is  a set of  pairwise
adjacent vertices. Note that the empty set is a clique.

\paragraph{Subgraphs, minors, topological minors and star minors}
Let $G$ be a graph on the vertex set $V(G)$ and with the edge set $E(G)$.
For $v\in V(G)$ and $e\in E(G)$, we denote by $G-v$ the graph
obtained from $G$ by removal of $v$ and all edges incident with $v$
and by $G\setminus e$ the graph obtained by
removal of $e$ from $G$.
For $\deg_{G}(v)=2$, we call by the \emph{dissolution of $v$ in $G$}
the graph obtained from $G$ by adding an edge connecting the neighbors
$N_{G}(v)$  of $v$ (if there is no such an edge in $G$) and then
by removing $v$.
The result of the \emph{contraction} of $e=\{x,y\}$ from $G$
is the graph $G/ e$ obtained from $G-x-y$ by
adding a new vertex $v_{x,y}$ and making it adjacent
to all  vertices of $(N_{G}(x)\cup N_{G}(y))\setminus \{x,y\}$.

For graphs $G$ and $H$,  we say that $H$ is an
\emph{induced subgraph of} $G$, and denote it by $H\subseteq_{is}G$,
if $H$ can be obtained from $G$ after  a sequence
of vertex removals. Also, for $S\subseteq V(G)$, we call $H$ the subgraph
of $G$ \emph{induced by $S$}, and write $H=G[S]$, if the vertex set
required to be removed from
$G$ in order to transform $G$ to $H$ is $V(G)\setminus S$.
%We call a class ${\cal G}$ of graphs \emph{hereditary} if it is closed under
%deleting vertices; in other words, if $G\in{\cal G}$ and $H$ is an
% induced subgraph of $G$, then $H\in {\cal G}$.

 We say that $H$ is a \emph{subgraph} of $G$,
 % and denote it as $H\subseteq_{sb}G$,
 if $H$ can be obtained from $G$  after applying a
sequence of vertex and edge removals. We say  that $H$ is a
\emph{topological minor} of $G$, %and denote it as $H\leq_{tm}G$,
if $H$ can be obtained from $G$  by applying a sequence
of vertex/edge removals and dissolutions. Finally, we say that
$H$ is a \emph{minor} of $G$ %, and denote it as $H\leq_{mn}G$,
if $H$ can be obtained from $G$  after applying a sequence of vertex
removals or edge removals/contractions.

The \emph{greatest reduced average degree   with rank $p$} of a graph $G$
is
\[
\nabla_{p}(G)=
\max\frac{|E(H)|}{|V(H)|} 
,\]
where maximum is taken over all the minors $H$ of $G$ obtained by contracting a set of vertex-disjoint subgraphs with
radius at most $p$ and then deleting any number of edges \cite{NO2008a,NO2008b, NO2008c}.  In this work, we consider only graphs with $p=1$. We say that 
 a graph $H$ is a {\em star minor} of $G$  if $H$ is obtained from $G$ by contracting edges of vertex-disjoint subgraphs of radius $1$ (or equivalently, vertex-disjoint stars).  
%of $G$ if in $G'$ we consider a packing of stars (a collection of vertex disjoint 
%copies of $K_{1,r}$) and then contract the edges of these stars.
 Thus   $\nabla_{1}(G)$ is the maximum density 
among all star-minors of $G$. % ($\nabla_{1}(G)$ has been defined in~\cite{NO2008a} -- see also~\cite{NO2008b, NO2008c}). 
 We also say that a graph $G$ 
is {\em $d$-degenerate} if each of its subgraphs (including $G$ itself) has a vertex of degree  at most $d$.
It is easy to observe that every graph $G$ is $2\cdot \nabla_{p}(G)$-degenerate for every $p\geq 0$.
%
% Let us remark that  there is a natural ``ordering'' of relations % the ordering
% $
% \subseteq_{is}$, $\subseteq_{sb}$, $\leq_{tm}$, $\leq_{mn}$ %defines a total
% %ordering
% in a sense
% that for any two relations $\leq_{1}$ and $\leq_{2}$ in this sequence
% where $\leq_{1}$ is on the left side of $\leq_{2}$,
% $H\leq_{1}G$ implies that $H\leq_{2} G$. As a result of this, every graph class closed under
% some of the relations $\subseteq_{sb},\leq_{tm},\leq_{mn}$ is hereditary.

%

%

\paragraph{Hypergraphs.}

A \emph{hypergraph} $H$ is a pair $(V(H),E(H))$ of a finite set $V(H)$,
called the
vertex set,
and a set $E(H)$ of subsets of $V(H)$, called the hyperedge set.
%We do not need to discuss simple hypergraphs.
%A hypergraph is \emph{simple} if all its hyperedges have arity at least 2.
The \emph{incidence graph}  of a hypergraph $H$
is the bipartite graph $I(H)$ on the vertex set $V(H)\cup E(H)$
such that $v\in V(H)$ is adjacent to $e\in E(H)$
in $I(H)$ if and only if $v$ is incident with $e$ in $H$ (in other words, $v\in
e$).
%The \emph{incidence matrix} of a hypergraph $H=(V,E)$
%is  a $V\times E$ matrix  $(a_{ve})_{v\in V,e\in E}$ over the binary field
%such that $a_{ve}=1$ if and only if $v$ is incident with $e$ in $H$.

%Let ${\cal H}$ be a class of simple hypergraphs. We say that
%$f: \Bbb{N}\rightarrow \Bbb{N}$
%is a \emph{density bounding function} for ${\cal H}$ if for
%every $H\in {\cal H}$,  $|E(H)|\leq f(|V(H)|)$.

\paragraph{Bipartite graphs.}
For a graph $G$ and a subset $X\subseteq V(G)$, we use
notation $\overline{X}$ for
$V(G)\setminus X$.
%A graph is \emph{bipartite} is it can be properly coloured with two colours.
For a bipartite graph $G$ with bipartition $X$ and $\overline{X}$,
% we use the notation $V(G)=X\cup \overline{X}$ to make
%clear that $X$ and $\overline{X}$ is the corresponding  bipartition.
%For a bipartite graph $G=(X\cup \overline{X}, E)$
its
\emph{bipartite adjacency matrix} is
a $|X|\times |\overline{X}|$ matrix
\[
\B_{G}=(b_{i,j})_{i\in X,j\in \overline{X}},\]
over the binary
field $\mathrm{GF}(2)$
such that $b_{i,j}=1$ if and only if $\{i,j\}\in E(G)$.

For a nonempty  subset $X$ of the vertex set of $G$,
we define the subgraph
$G\langle X\rangle$
with vertex set $V(G)$ and edge set
\[
\{\{x,x'\}\in E (G) \mid x\in X, x'\in \overline{X}\}.\]
Hence $G\langle X\rangle $
is the bipartite subgraph of $G$ that contains only the edges
with endpoints in  $X$ and $\overline{X}$.

%%% FF CHECK if Sang-il's notations are better?

%Given a bipartite graph $G=(X\cup \overline{X},E)$, we call two vertices $x,y\in \overline{X}$
%\emph{twins} if they
%have the same neighborhood. i.e. $N_{G}(x)=N_{G}(y)$.

\paragraph{Rank-width.}
For a graph $G$ and $X\subseteq V(G)$, the \emph{cut-rank} function is
defined to be
$$\cutrk_G(X)=\rank (\B_{G\langle X\rangle}).$$
If $X=\emptyset$ or $X=V(G)$, then $\cutrk_G(X)=0$.
Let us note that $\B_{G\langle X\rangle}$ is a matrix over the binary field
% and thus the
when we consider 
rank function of this matrix. % is also over the binary
%field.

A tree is \emph{ternary} if all its vertices are
of  degree 1 or 3.
We denote by $L(T)$ the set of leaves of a tree $T$.
A \emph{rank-decomposition} of a graph $G$ is a pair $(T,\mu)$
consisting of a ternary tree $T$
and a bijection $\mu:V(G)\rightarrow L(T)$.
Each edge $e$ of $T$ defines a partition $(X_e,Y_e)$ of $L(T)$.
The \emph{width} of an edge $e$ of $T$
is $\cutrk_G(\mu^{-1}(X_e))$. The \emph{width} of a rank-decomposition
$(T,\mu)$ is the maximum width of all edges of $T$.
The \emph{rank-width}  of a graph $G$, denoted by $\rwd(G)$,
is the minimum width of all
rank-decompositions of $G$.
If $\abs{V(G)}\le 1$, then $G$ admits no rank-decompositions from the above
definition. If that is the case,  we define
the rank-width of $G$ to be $0$.

\paragraph{Tree-width.}
A \emph{tree decomposition} of  a graph $G$
is a pair $(T,X)$, where $T$
is a tree, and
$X=(\{X_{v} \mid v\in V(T)\})$ is a collection of
subsets of $V(G)$ such that
\begin{enumerate}[(T1)]
\item For each edge $e$ of $G$,
 the endpoints of $e$ are contained in $X_v$ for some $v\in V(T)$.
\item If $a,b,c\in V(T)$ and the path from $a$ to $c$ in $T$ contains $b$,
 then $X_a\cap X_c\subseteq X_b$.
\item $\cup_{v\in V(T)} X_v= V(G)$.
\end{enumerate}
The \emph{width} of a tree decomposition $(T,(X_v)_{v\in V(T)})$
is $\max_{v\in V(T)} \abs{X_v}-1$.
The \emph{tree-width} of a graph is the minimum width of all
tree decompositions of the graph.

\paragraph{Clique-width.}
For a positive integer $k$, a \emph{$k$-graph} is a pair $(G,\lab)$
of a graph $G$ and a labeling function
\[
\lab:V(G)\rightarrow \{1,2,\ldots,k\}.
\]
If $\lab(v)=i$, then we call $i$ the \emph{label} of $v$.
From now on, we define \emph{$k$-expressions}, which are algebraic
expressions with the following four  operations
to describe how to construct $k$-graphs.
\begin{itemize}
\item
 For $i\in \{1,2,\ldots,k\}$, $\cdot_i$ is a $k$-graph consisting of a
 single vertex of label $i$.
\item
 For distinct $i,j\in \{1,2,\ldots,k\}$,
 $\rho_{i\to j}(G,\lab)=(G,\lab')$
 in which
 $\lab'(v)=\lab(v)$ if $\lab(v)\neq i$
 and $\lab'(v)=j$ if $\lab(v)=i$.
\item For distinct $i,j\in \{1,2,\ldots,k\}$,
 $\eta_{i,j}(G,\lab)=(G',\lab)$
 in which
 $V(G')=V(G)$
 and $E(G')=E(G)\cup\{vw: \lab(v)=i,\lab(w)=j\}$.
\item $\oplus$ is the disjoint union of two $k$-graphs.
 In other words, $(G_1,\lab_1)\oplus (G_2,\lab_2)=(G,\lab)$
 in which $G$ is the disjoint union of $G_1$ and $G_2$,
 and $\lab(v)=\lab_1(v)$ if $v\in V(G_1)$
 and $\lab(v)=\lab_2(v)$ if $v\in V(G_2)$.
\end{itemize}
The \emph{clique-width} of a graph is the minimum $k$
such that
there exists a $k$-expression
with value $(G,\lab)$ for some labeling function
$\lab$.

\section{Rank-width and clique-width}\label{sec:rwdcwd}
For a graph $G$ and a set $X\subseteq V(G)$, we define
\[\lambda_{G}(X)=\lvert\{ N_{G\langle X\rangle}(v)\mid v\in \overline{X}\}\rvert,\]
which is the number of distinct neighborhoods
of the vertices in $\overline{X}$ in the graph $G\langle X\rangle$.
By the
definition of $G\langle X\rangle$, each such a neighborhood is
a subset of $X$.
For integer $k>0$,
we also  define  $$\lambda_{G}(k)=\max\{\lambda_{G}(X)\mid X\subseteq V(G), |X|\le k\}.$$
Clearly, in general, $\lambda_{G}(k)\le 2^{k}$. As we will see in the following sections,
better bounds can be obtained when $G$ belongs to certain  graph classes.

\begin{LEM}\label{lem:rank}
  Let $G$ be a graph and let $X$ be a subset of $V(G)$ such that
  $\rho_{G}(X)\le k$. Then the bipartite
 adjacency matrix of $G\langle X\rangle$ has
 at most $\lambda_{G}(k)$ distinct  rows.
\end{LEM}
\begin{proof}
Let $M$ be the bipartite adjacency matrix of $G\langle X\rangle$.
We may assume that $M$ has exactly $\rho_G(X)$ columns, because there
exist $\rho_G(X)$ columns whose linear combination spans
all other column vectors.
%Let $M'$ be a submatrix of $M$ formed by selecting maximal linearly
%independent columns of $M$.
%We choose $\rho_{G}(X)$ columns in $M$ such that each of the remaining
%column is a linear combinations of them.
%Notice that if two rows in $M'$ are equal, then the corresponding
%rows in $M$ are also  equal and thus it is sufficient
%to prove the assertion of the lemma
%for the submatrix $M'$
%of $M$ formed by these  $\rho_{G}(X)$ columns.
%Let  $Y$ be the vertices of $\overline{X}$ corresponding to the
%columns of $M$. Notice that distinct rows in $M'$
%correspond to distinct neighborhoods in $G\langle X\rangle$  of the
%vertices of $\overline{X}$ in $Y$ that is upper bounded by
%$\lambda_{G}(X)$ and thus by $\lambda_{G}(k)$.
\end{proof}

The following lemma is implicit in Oum and Seymour~\cite{OS2004}.
For a set $X$ of vertices of a graph $G$, let $c_G(X)$
be the number of  distinct \emph{nonzero} rows in the bipartite
adjacency matrix of $G\langle X\rangle$. For  a rank-decomposition
$(T,\mu)$ of $G$, % of width at most $k$,
we define $\beta_{G}(T,\mu)=\max\{ \max\{c_G(X_{e}),c_G(Y_e)\}\mid  e\in E(T)\}$.

%$\A_G[X,V(G)\setminus X]$.

\begin{LEM}\label{lem:cwdrwd}
 %Let $G$ be a graph
 %and
 Let $(T,\mu)$ be a rank-decomposition of a graph $G$. % of width at most $k$.
 Then
 the clique-width of $G$ is at most $2 \beta_{G}(T,\mu)+1$.
\end{LEM}
\begin{proof} We set $C=\beta_{G}(T,\mu)$.
 We may assume that $\abs{V(G)}\ge 3$.
 %and that the degree of every internal node of $T$ is 3.
We turn $T$ into a rooted directed tree by choosing an internal vertex $r$
as a root and by directing all edges from the root.
%
%  Let $r$ be an internal node of $T$.
% We direct every edge of $T$ out of $r$ so that $T$ is a rooted
% binary tree.

 For a vertex $v$ in $T$, let
 $D_v=\{x\in V(G): \text{$\mu(x)$ is a descendant of $v$ in $T$}\}$.
 Let $G_v$ be the subgraph of $G$ induced on $D_v$.

 We  claim that for each vertex $v$ of $T$,
 there is a $(2C+1)$-expression $t_v$ with value $(G_v,lab_v)$ for some
 map $lab_v:V(G_v)\rightarrow \{1,2,\ldots,C,2C+1\}$
 satisfying the following two conditions:
 \begin{itemize}
 \item   If
 $\lab_v(x)=\lab_v(y)$,
 then every vertex in $V(G)\setminus D_v$
 is  either adjacent to both $x$ and $y$,
 or nonadjacent to both $x$ and $y$.
 \item  If $x$ in $D_v$ has no neighbor in $V(G)\setminus D_v$,
   then $\lab_v(x)=2C+1$.
 \end{itemize}

 We proceed by induction on the number of descendants of $v$ of $T$. If $v$
 is a leaf, then we let $t_v=\cdot_{2C+1}$.
 Now let us assume that $v$ has two children $v_1$ and $v_2$.
 By the induction hypothesis,
 we have $(2C+1)$-expressions $t_{v_1}$ and $t_{v_2}$
 with values   $(G_{v_1},\lab_{v_1})$, $(G_{v_2},\lab_{v_2})$,
 respectively.
 We glue $t_{v_1}$ and $t_{v_2}$ to obtain a $(2C+1)$-expression $t_v$
 for $G_v$.
 Let $F$ be the set of pairs $(i,j)$ such that
 there exist a vertex $x\in D_{v_1}$ and a vertex $y\in D_{v_2}$
 such that $\lab_{v_1}(x)=i$, $\lab_{v_2}(y)=j$, and $x$ is adjacent
 to $y$ in $G$.
 Let $N$ be the set of integers $i\in \{1,2,\ldots,2C\}$ such that
 there exists  a vertex $v$ of label $i$ in $D_{v_1}$
   or a vertex $v$ of label $(i-C)$ in $D_{v_2}$
   such that $v $ has no neighbors in $V(G)\setminus D_v$.
 Then let
 \[t^*=
 (\mathop{\circ}_{i\in N}\rho_{i\to 2C+1})
 ((\mathop{\circ}_{(i,j)\in F} \eta_{i,j+C})(t_{v_1}\oplus
 \rho_{1\to C+1}(\rho_{2\to C+2}(\cdots (\rho_{C\to 2C}(t_{v_2}))\cdots)))).\]
 Then $t^*$ is a $(2C+1)$-expression with value $(G_v,\lab^*)$ say.
 So far, $\lab^*$ satisfies the condition that if two vertices in
 $D_v$
 have  the same $\lab^*$ value, then they have the identical set of
 neighbors out of $D_v$.

 Since $\cutrk_G(D_v)\le k$ and
 $c_G(D_v)\le C$, 
 there are at most $C$ distinct vertices in $D_v$
 having some neighbors in $V(G)\setminus D_v$.
 We obtain  a $(2C+1)$-expression $t'$ from $t^*$ by
 applying
 $\rho_{i\to    j}$
 to merge two labels $i,j$
 whenever vertices of $i$ and $j$ have the same nonempty set of
 neighbors
 in $V(G)\setminus D_v$.
 Let $(G_v,\lab')$ be the value of $t'$.
 Then $\lab'$ has at most $C+1$ distinct values.

 Let $t_v$ be a $(2C+1)$-expression obtain from $t'$ by applying
 $\rho_{i\to j}$ operations whenever $2C\ge i>C\ge j$ and there are no vertices
 of label $j$. Then $t_v$ is what we wanted.
 This proves the induction claim.

 Now $t_r$ is a $(2C+1)$-expression of $G$ and therefore the clique-width
 of $G$ is at most $2C+1$.
\end{proof}
\begin{LEM}\label{lem:cwdrwd2}
 Let $G$ be a graph with at least one edge.
 %If $\abs{E(G)}>0$,
  Then
 \[\rwd(G)\le \cwd(G)\le 2 \lambda_{G}(\rwd(G))-1.\]
\end{LEM}
\begin{proof}
 Let $\rwd(G)\le k$ and let $(T,\tau)$ be a rank-decomposition of $G$
of  width  at most $k$. Since $\abs{E(G)}>0$, we have that $k>0$.
For every   $e\in E(T)$,
  the rank of the bipartite adjacency matrix $M_{e}$ of
  $G\langle \tau^{-1}(X_{e})\rangle$ is at most
 $k$.  If $\rank(M_{e})=0$, then $c_G(X_{e})=0$.
 Now let us assume that $\rank(M_{e})>0$.
 Let
 $M'_{e}=\left(\begin{smallmatrix}M_{e} \\0\end{smallmatrix}\right)$
 be  the matrix obtained by adding a zero row to $M_{e}$.
 By Lemma~\ref{lem:rank},
 $M'_{e}$ has at most $\lambda_{G}(k)$ distinct rows.
 Then $M_{e}$ has at most $\lambda_{G}(k)-1\geq 0$ nonzero distinct rows.
 In any case, we deduce that $c_G(X_{e})\le \lambda_{G}(k)-1$ and thus $\beta_{G}(T,\tau)\le \lambda_{G}(k)-1$.
 By Lemma~\ref{lem:cwdrwd}, we deduce that
 $\cwd(G)\le 2\cdot \lambda_{G}(\rwd(G))-1$.
\end{proof}

Lemma~\ref{lem:cwdrwd2} along with the fact that
$\lambda_{G}(k)\le 2^{k}$ yields
the exponential upper bound in~\eqref{eq:rwdcw}.
In general, such a bound is unavoidable
because  Corneil and Rotics~\cite{CR2005}
showed that,
for each $k$, there is a graph $G_k$ such that
$\cwd(G_k)\ge 2^{\lfloor k/2\rfloor-1}$
and
$\twd(G_k)= k$,
which implies
$\rwd(G_k)\le k+1$ by \eqref{eq:rwdtwd}.
In the following sections we refine the
bound in~\eqref{eq:rwdcw} for certain graph classes.
Our main tool is to derive
better estimations of the function~$\lambda_{G}$.

\section{Graphs with no complete graph minor}\label{sec:minor}
Our goal of this section is to 
prove that, for a fixed $r>2$, the tree-width, 
the rank-width and the clique-width of a graph with no $K_r$-minor
are within a constant factor, where the constant only depends on $r$.
We also aim to make this section as a reference to be used later for
other graph classes.

\medskip

%Let $r$ be fixed and let $\mathcal H$ be the class of graphs with no $K_r$-minor.
Let us consider the following  problems for a fixed positive integer
$r$.
\begin{enumerate}[P1:]
\item
  Does there exist a constant $c_1$ 
  such that, for all $n>0$, 
  every $n$-vertex graph %in $\mathcal H$
  has at most $c_1n$ edges
  if it has no $K_r$-minor?
\item
  Does there exist a constant $c_2$ 
  such that, for all $n>0$,  
  every $n$-vertex graph %in $\mathcal H$ 
  has at most $c_2n$   cliques
  if it has no $K_r$-minor?

\item
  Does there exist a constant $c_3$
  such that, for all $n>0$,  every $n$-vertex hypergraph
  has at most $c_3n$ hyperedges
  if its incidence graph %is in $\mathcal H$?
  has no $K_r$-minor?
\item
  Does there exist a constant $c_4$
  such that, for all $n>0$,
  every binary matrix of rank $n$ has at most $c_4n$  distinct rows
  if the bipartite graph having the matrix as a bipartite adjacency
  matrix %is in $\mathcal H$?
  has no $K_r$-minor?

\item
  Does there exist a constant $c_5$ such that, for all $n>0$, 
  the tree-width of every graph % in $\mathcal H$
  of rank-width $n$
  is at most $c_5n$
  if the graph has no $K_r$-minor?
\end{enumerate}
Note that these problems are trivial if $r\le 2$
and therefore we will assume that $r>2$.

The problem P1 was answered by Kostochka
\cite{Kostochka1982,Kostochka1984}
and Thomason \cite{Thomason1984} independently. Later Thomason
determined the exact constant as follows.
\begin{PROP}[P1;{Thomason \cite{Thomason2001}}]\label{prop:thomason}
  There is a constant $\alpha$ such that
  every $n$-vertex graph with no $K_r$-minor
  has at most $(\alpha r \sqrt{\log r})n$ edges.
  Moreover, this result is tight up to the value
  of  $\alpha=0.319\ldots+o(1)$.
\end{PROP}
This proposition implies that $c_1=\alpha r\sqrt{\log r}$ satisfies
$c_1$.
Now we will explain that any upper bound of $c_i$
will give upper bounds for $c_{i+1}$.
Moreover, our proof technique can be applied to classes of
graphs more general than graphs with no $K_r$-minor which we will
discuss later.

To answer P2, we claim that
every $n$-vertex graph with no
$K_r$-minor will have at most $2^{O(r\sqrt{\log r})}n $ cliques.
To see this, we use a simple induction argument by counting
cliques  containing a vertex $v$ of the minimum degree
to   show that every $n$-vertex graph with no
$K_r$-minor has at most $2^{2c_1}n$ cliques if $c_1\ge 1/2$.
More precisely, one can prove that
if an $n$-vertex graph  is $d$-degenerate and $n\ge d$,
then it has at most $2^d(n-d+1)$ cliques, see Wood \cite{Wood2007}.
We now aim to  show that the above bound on the number of cliques
can be improved to $2^{O(r\log\log r)}n$. 

\begin{LEM}\label{lem:cliqueminor}
  There is a constant $\alpha$ such that, for $r\ge 2$, 
   every $n$-vertex graph with no $K_r$-minor
  has at most $\frac1{r+1}\binom{r+1}{k}{(2\alpha\sqrt{\log r})}^{k-1}n$
  cliques of size $k$
  for $1\le k\le r-1$.
\end{LEM}
\begin{proof}
  Let $G$ be an $n$-vertex graph with no $K_r$-minor.
  We take $\alpha$ from Proposition~\ref{prop:thomason}.
  We apply induction on $r$. If $r=2$ or $k=1$, then it is trivial.
  So we may assume that $r>2$ and $k>1$.
  For a vertex $v$, the subgraph induced on
  the neighbors of $v$ contains at most $\frac{1}{r} \binom{r}{k-1}
  (2\alpha \sqrt{\log(r-1)})^{k-2}\deg (v)$ cliques of size $k-1$
  because it has no  $K_{r-1}$-minor.
  Since each clique of size $k$ is counted $k$ times,   $G$ has
  at most $\frac{1}{k} \sum_{v\in V(G)} \frac{1}{r} \binom{r}{k-1}
  (2\alpha \sqrt{\log(r-1)})^{k-2}\deg (v)$ cliques of size $k$.
  The conclusion follows because
  $\sum_{v\in V(G)}\deg(v)\le (2\alpha r\sqrt{\log {r}})n$
  by Proposition~\ref{prop:thomason}
  and 
  $\binom{r+1}{k}=\frac{r+1}{k}\binom{r}{k-1}$.
\end{proof}
\begin{PROP}[P2]\label{prop:p2minor}
  There is a constant $\mu$ such that,
  for $r>2$, 
  every $n$-vertex graph with  no $K_r$-minor has at most
  $n2^{\mu r\log \log r}$ cliques.
\end{PROP}
\begin{proof}
  Let $\alpha$ be the constant in  Proposition \ref{prop:thomason}.
  We may assume that $\alpha\ge 0.5$ by taking a larger value if necessary. (It is likely that $\alpha$ is bigger than $0.5$ if we want it to be satisfied by all graphs, not just large graphs.)
 Since $\log r\ge 1$, we have that 
 the number of cliques of size $i$ 
 is at most $n\frac{1}{r+1}\binom{r+1}{i}(2\alpha)^{i-1} (\sqrt {\log r})^{r-1}
 \le n\binom{r}{i-1} (2\alpha)^{i-1} (\sqrt{\log r})^{r}$
  when $1\le i\le r-1$
  by Lemma~\ref{lem:cliqueminor}.
  Let $C$ be the number of cliques in the graph. Then we obtain the following.
  \begin{align*}
    C&\le 1+n(\sqrt{\log r})^{r}\sum_{i=1}^{r-1}\binom{r}{i-1} (2\alpha)^{i-1} \\
    &\le n(\sqrt{\log r})^{r}\sum_{i=0}^{r}\binom{r}{i} (2\alpha)^{i}
   &\text{because }(2\alpha)^r\ge 1\\
   &= n(\sqrt{\log r})^{r} (1+2\alpha)^r.
 \end{align*}
  Let $c=\frac{\log(1+2\alpha)}{\log\log3}$. Then
  $\log(1+2\alpha)\le \log \log r$
  and therefore
  $C\le n 2^{(\frac12+c)\frac1{\log 2} r\log\log r}$.
\end{proof}
\begin{PROP}[P3]\label{prop:p2p3}
  Let $c_2$ be a constant satisfying P2.
  Then every $n$-vertex hypergraph has at most $c_2n$ hyperedges
  if its incidence graph has no $K_r$-minor; therefore
  $c_3= c_2$ satisfies P3.
\end{PROP}
\begin{proof}
  % Let us now prove that $c_3\le c_2$.
  Let $H$ be a hypergraph with $n$ vertices whose incidence graph
  $I(H)$ has no $K_r$-minor. We may assume that every subset of a
  hyperedge $e$ of $H$ is a hyperedge of $H$, because otherwise we may replace
  $e$ by its proper subset. Let $G$ be a graph on $V(H)$ obtained from
  $H$ by   deleting all hyperedges of arity other than $2$. It is easy to observe
  that $G$ is a minor of $I(H)$ (actually, $G$ is a topological minor or, even better, 
  a star minor of $I(H)$) and therefore $G$ has no $K_r$-minor.
  Moreover for each hyperedge $e$ of $H$, $G$ has a corresponding clique
  on the same vertex set. Thus, the number of hyperedges of $H$ is at most
  $c_2n$.
\end{proof}
\begin{PROP}[P4]\label{prop:p3p4}
  Let $c_3$ be a constant satisfying P3. 
  Then every binary matrix of rank $n$ has at most $c_3n$ distinct
  rows
  if the bipartite graph having the matrix as a bipartite adjacency
  matrix has no $K_r$-minor;
  this implies that
  $c_4= c_3$ satisfies P4.
\end{PROP}
\begin{proof}
  %To see $c_4\le c_3$:
  Let $M$ be a binary matrix of rank $n$.
  Let $G$ be the bipartite graph having $M$ as a bipartite adjacency matrix.
  We claim that $M$ has at most $c_3n$ distint rows.
  We may assume that $M$ has $n$ columns by deleting linearly
  dependent columns. We may also assume that $M$ has no identical
  rows. Then let $H$ be a
  hypergraph such that its  incidence graph is $G$
  and the vertices  of $H$ correspond to  vertices of $G$
  representing the columns of $M$.
  (Note that in this paper, a hypergraph has no parallel edges.)
  Since $G$ has no $K_r$-minor, $H$ has at most $c_3n$ hyperedges
  and therefore $M$ has at most $c_3n$ rows.
\end{proof}
\begin{PROP}[P5]\label{prop:p4p5}
  Let $c_4$ be a constant satisfying P4.
  If $G$ is a graph 
  with no $K_r$-minor,
  then
  $\cwd(G)\le 2c_4\rwd(G)-1$ and
  $\twd(G)+1\le3(r-2)(2c_4\rwd(G)-1)$.

  Therefore $c_5=6(r-2)c_4$ satisfies P5.
\end{PROP}
\begin{proof}
  Let $G$ be a graph of rank-width at most $n$ with no $K_r$-minor.
  We will only need the following two facts:
  \begin{itemize}
  \item $G$ has no $K_{r-1,r-1}$
    as a subgraph.
  \item Every bipartite subgraph of $G$ has no $K_r$-minor.
  \end{itemize}
  
  First we claim that the clique-width is at most $2c_4n-1$.
  By Lemma~\ref{lem:cwdrwd2},
  it is enough to  prove that $\lambda_G(n)\le c_4n$.
  Let $X$ be a subset of at most $n$ vertices of $G$. (Here, $n$ is
  the rank-width of $G$.)
  Let $M$ be the bipartite adjacency matrix of $G$ whose rows and
  columns are indexed by $V(G)\setminus X$ and $X$, respectively.
  Then obviously $\rank(M)\le n$. Moreover the bipartite graph having
  $M$ as a bipartite adjacency matrix has rank at most $n$.
  Thus $M$ has at most $c_4n$ distinct rows and therefore $\lambda_G(n)\le
  c_4n$.
  This proves the claim.

  Since $G$ has no $K_r$-minor,
  $G$ does not contain $K_{r-1,r-1}$ as a subgraph.
  By \eqref{eq:gw},
  the tree-width of $G$ is at most
  $3(r-2)(2c_4n-1)-1$. 
\end{proof}
Let us summarize what we have for graphs with no $K_r$-minor.
\begin{THM}\label{thm:minor}
  There is a constant $\mu$ such that
  for each integer $r> 2$,
  if $G$ is a graph with no $K_r$-minor, then
  \begin{align*}
    \cwd(G)&< 2\cdot 2^{\mu r\log\log r}\rwd(G),\\
  \twd(G)+1&< 6(r-2)2^{\mu r \log \log r} \rwd(G).
  \end{align*}
\end{THM}
\section{Graphs of bounded genus}
\label{sec:planar}

%The \emph{Euler genus} of
%a nonorientable surface $\Sigma$
%is equal to the nonorientable genus
%$\tilde{g}(\Sigma)$ (or the crosscap number).
%The \emph{Euler genus}  of an orientable   surface
%$\Sigma$ is $2{g}(\Sigma)$, where ${g}(\Sigma)$ is  the orientable genus
%of $\Sigma$.
If $\Sigma$ is a surface %$\Sigma$
which can be obtained from the
sphere by adding $k$ crosscaps and $h$ handles,
then \emph{Euler genus} $\genus(\Sigma)$ of the surface $\Sigma$
is $k+2h$.
We refer to the book of Mohar and Thomassen \cite{MT2001}
for
more details  on graph embeddings.
\emph{Euler genus} $\genus(G)$ of a graph  $G$
is the minimum $r$ such that the
graph can be embedded into a surface of Euler genus~$r$.

A hypergraph is \emph{planar} if its  incidence graph is
planar, see Zykov~\cite{Zykov1974}. Also a hypergraph is \emph{embeddable}
on a surface of Euler genus $r$ if so is its incidence graph.
For formal definitions of hypergraph embeddings
on surfaces (called ``paintings'')
see~\cite{RS1994}.
\emph{Euler genus} $\genus(H)$
of a hypergraph $H$ is the minimum $r$ such that
$H$ can be embedded into a surface of
Euler genus $r$.

For graphs of Euler genus at most $r$, Euler's formula allows us to answer P1
easily; every $n$-vertex graph of Euler genus  $r$ has at most $3n-6+3r$ edges
if $n\ge 3$.
We may
obtain easy answers to  P2 and P3
by using the fact that such graphs have vertices of small degree.
However, that approach will
give us the following: the number of hyperedges
of a hypergraph 
of Euler genus
at most $r$ is at most $64n+f(r)$ for some function~$f$.
In the next lemma, we improve $64$
to $6$ for P3.

We remark that Wood~\cite{Wood2007} showed that an $n$-vertex planar
graph has at most $8(n-2)$ cliques if $n>2$; This answers P2 for
planar graphs. However, for P3, we can improve $8(n-2)$ to $6n-9$ by the
following proposition.
As a generalization of Wood~\cite{Wood2007},
Dujmovi{\'c} et al.~\cite{DFJW2009} showed that an $n$-vertex graph embedded on a surface
has at most $8n+\frac32 2^{\omega}+o(2^\omega)$,
where $\omega$ is the maximum integer such that the complete graph $K_\omega$ can be embedded on the same surface. Notice that
their bound can also be used to answer P3 but 
our bound for P3 improves their $8n+O(1)$ to $6n+O(1)$ for a fixed surface.

\begin{PROP}[P3]
\label{prop:surface}
Let $H$ be an $n$-vertex hypergraph 
embeddable on a surface of Euler genus $r$, $r>0$, $n>2$.
Then
$H$ has at most $(6n-9+5r)$ hyperedges.%, unless $r=0$ and $n\le2$.
%For $r\geq 0$,  $f_{r}(k)= 5(k-2+r)$ is the density bounding function for
%the class of simple hypergraphs of Euler genus at most $r$.
\end{PROP}
\begin{proof}
  Since $|E(H)|\le 2^n$,  we may  assume that $n\ge 3$.
  
  We assume that the incidence graph of $H$ is embedded on
  a surface $\Sigma$ of the Euler genus $r=\genus(H)$.
  % Let $H$ be a simple hypergraph embedded in  a surface $\Bbb{S}$
  % with Euler genus $r$.
  There is at most $1$ hyperedge of arity $0$ because 
  $\emptyset$ is the only possible hyperedge of  arity $0$.
  It is also trivial that there are at most $n$ hyperedges of arity
  $1$.

  We now count hyperedges of arity at least $2$.
  We define sub-hypergraphs $H_2$ and $H_{\geq 3}$ of $H$ such that
  \[
  V( H_{2})=V(H), \text{ and } E(H_2)=\{e\in E(H)\mid |e|=2\}),\] and
  \[V(H_{\geq 3})= V(H) \text{ and } E(H_{\geq 3})=\{e\in E(H)\mid |e|\geq 3\}).\]
  In other words,
  $H_{2}$ contains the hyperedges of $H$ of arity $2$ and $H_{\geq 3}$
  contains
  the hyperedges of arity greater than $2$.
  Clearly, both $H_{2}$ an $H_{\geq 3}$ are hypergraphs embeddable
  on $\Sigma$.
  Because  $H_{2}$ has no parallel edges or loops,
  by Euler's formula, we
  have  $|E(H_{2})|\le 3n-6+3r$.
  To bound the number of  hyperedges in $H_{\ge 3}$, we construct a
  graph $F$ as follows.
  For each  hyperedge $e=\{v_{1},v_{2},\ldots,v_{l}\}$ whose endpoints
  are cyclically
  ordered as $v_{1},v_{2},\ldots,v_{l},v_{1}$ in the surface, we remove
  $e$ and
  add edges $\{v_{1},v_{2}\},\ldots,\{v_{l-1},v_{l}\},\{v_{l},v_{1}\}$.
  We will not create parallel edges or loops.
  Then  each hyperedge of $H_{\geq 3}$
  corresponds to a face of the embedding of $F$ in $\Sigma$
  and no two hyperedges are mapped to the same face in $F$.
  The graph~$F$ has $n$ vertices, and, again by  Euler's formula, we derive
  that $|E(H_{\geq 3})|\le 2k-4+2r$.
  So we conclude that $|E(H)|\le 1+n+ (3n-6+3r)+(2n-4+2r)=6n-9+5r$.
\end{proof}

Proposition \ref{prop:surface} is tight; Given any plane triangulation,
we attach a hyperedge of arity $3$ for each triangle. Then we
obtain $6n-9$ hyperedges in the planar hypergraph.

To answer P4 for graphs of Euler genus at most $r$, it is fairly
straightforward to apply the same argument of
Proposition~\ref{prop:p3p4} to deduce that
every binary matrix of rank $n$ has at most $(6n-9+5r)$ distinct rows if
the matrix induces a bipartite graph whose 
Euler genus is at most $r$.

Finally let us consider the problem P5 for graphs of Euler genus at
most $r$.
To mimic the argument of Proposition~\ref{prop:p4p5},
we need to determine the largest complete bipartite graphs with Euler
genus at most $r$.
Ringel \cite{Ringel1965,Ringel1965a} showed that
if $m,n\ge 2$, then
the orientable  genus of $K_{m,n}$ is $\lceil  (m-2)(n-2)/4 \rceil$
and if $m,n\ge 3$, then the nonorientable genus of $K_{m,n}$ is
$\lceil (m-2)(n-2)/2\rceil$.
It follows that if $t>2+\sqrt{2r}$, then
$K_{t,t}$ is not embeddable on a surface of Euler genus $r$.

\begin{THM}\label{thm:surface}
  Let $G$ be a graph embeddable on a surface of Euler genus $r$.
  Then
  \begin{align*}
    %\cwd(G)&\le 2\max(4,6\rwd(G)-9+5r)-1,\\
    %\twd(G)&\le 3(2+\sqrt{2r})\max(4,(6\rwd(G)-9+5r))-1.\\
    \cwd(G)&< 12\rwd(G)+10r,\\
    \twd(G)+1&< 3(2+\sqrt{2r})(6\rwd(G)+5r).
 \end{align*}
\end{THM}
\begin{proof}
  Let $t$ be a minimum integer such that $t>2+\sqrt{2r}$.
  Then $K_{t,t}$ is not embeddable on the surface of Euler genus $r$
  and therefore
  $G$ has no $K_{t,t}$ subgraph. By Gurski and Wanke's inequality
  \eqref{eq:gw},
  we have
  \[
  \twd(G)+1\le  3(t-1)\cwd(G).
  \]

  %Let us first assume that $r>0$.
  From Proposition~\ref{prop:surface},
  we have $\lambda_G(n)\le 6n-9+5r$ unless $r=0$ and $n\le 2$.
  We use a relaxed inequality $\lambda_G(n)< 6n+5r$, true for all $r\ge 0$ and $n\ge 1$.
  Then $\cwd(G)< 12\rwd(G)+10r$ and 
  $\twd(G)+1< 3(2+\sqrt{2r})(6\rwd(G)+5r)$.
\end{proof}
\section{Graphs  excluding topological minors.}
\label{sec:tm}

% Notice that the proof of Claim~\ref{prop:surface} does not apply for the topological minor relation because
% its claim makes use of the contraction operation. To avoid this problem, we will make use the
% following proposition providing a density bounding function for graphs
% not containing $K_{r}$ as a topological subgraph.
We now relax our problems to graphs with no $K_r$ topological
minor. As we did in Section~\ref{sec:minor}, we begin by answering
P1; how many edges can a graph  have if it has no $K_r$ 
topological minor?

\begin{PROP}[P1; Bollob\'{a}s and Thomason \cite{BT1998};
%\footnote{Bollob\'{a}s $\&$ Thomason,
%Proof of a conjecture of Mader,
%Erd/H{o}s and Hajnalon topological complete subgraphs, Manuscript 1994.}
Koml\'{o}s and Szemer\'edi \cite{KS1996b}]
%1996\footnote{Koml\'{o}s $\&$ Szemeredi, Topological cliques in graphs II. Combinatorics Probability, and Computing, {\bf 5}, (1996)}]
\label{prop:anto}
There is a constant $\beta$ such that 
for every $r$, 
every graph of average degree at least $\beta r^2$ 
contains $K_{r}$ as a topological minor. %Moreover $\beta\le 1116$.
Subsequently every $n$-vertex graph with more than $\frac\beta2 r^2 n$
edges
contains $K_r$ as a topological minor.
\end{PROP}
Thomas and Wollan's Theorem \cite{TW2005} can be used to obtain that
$\beta=10$ satisfies the above proposition; see Diestel
\cite[Theorem 7.2.1]{Diestel2005} with the corrected proof in the
web site of Diestel\footnote{\url{http://diestel-graph-theory.com/corrections/3rd.edn.corrections.pdf}}.

If we use the fact that every graph with no $K_r$ topological minor
has a vertex of degree at most $\beta r^2$, 
we can easily show that every $n$-vertex graph with no $K_r$
topological minor can have at most $2^{\beta r^2}n$ cliques.
We aim to improve $2^{O(r^2)}n$ to $2^{O(r\log r)}n$
as we did in Proposition~\ref{prop:p2minor}.

\begin{LEM}\label{lem:cliquetopminor}
  Let $r\ge 2$. There is a constant $\beta$ such that
  every $n$-vertex graph with no $K_r$ topological minor
  has at most $\frac1{r+1} \binom{r+1}{k} (\beta r)^{k-1}n$ cliques of
  size $k$ for $1\le k\le r-1$.
\end{LEM}
\begin{proof}
  We take the same $\beta$ of Proposition~\ref{prop:anto}.
  We proceed by induction on $r$.
  Let $G$ be an $n$-vertex graph with no $K_r$ topological minor.
  We may assume that $k\ge 2$ and $r\ge 3$.
  For each vertex $v$, there are at most $\frac{1}{r}
  \binom{r}{k-1}(\beta (r-1))^{k-1}\deg(v)$ cliques of size $k$
  containing $v$.
  Since each clique of size $k$ is counted $k$ times,
  there are at most $\frac1{kr} \binom{r}{k-1}(\beta (r-1))^{k-1}
  (2|E(G)|)$ cliques of size $k$.
  By Proposition~\ref{prop:anto}, $2|E(G)|\le \beta r^2n$.
  The conclusion follows because
  $\frac1k\binom{r}{k-1}=\frac1{r+1}\binom{r+1}{k}$.
\end{proof}
\begin{PROP}[P2]\label{prop:cliquetopminor}
  There is a constant $\tau$ such that, for $r>2$, 
  every $n$-vertex graph with no $K_r$ topological minor
  has at most $2^{\tau r\log r}n$ cliques.
\end{PROP}
\begin{proof}
  Let $G$ be an $n$-vertex graph with no $K_r$ topological minor.
  Let $\beta$ be the constant in Proposition~\ref{prop:anto}.
  Since planar graphs have no $K_5$ topological minor,
  $\frac{25}2\beta\ge 3$ and so $\beta\ge \frac{6}{25}$.
  We may assume that $n\ge 3$ by assuming that $2^{3\tau \log 3}\ge 2$.
  By Lemma~\ref{lem:cliquetopminor}, $G$ has at most
  $C=1+\frac1{r+1}\sum_{k=1}^{r-1}\binom{r+1}{k}(\beta r)^{k-1}n$ cliques.
  \begin{align*}
    C&\le \frac4{3(r+1)} \sum_{k=1}^{r-1} \binom{r+1}k (\beta r)^{k-1}
    n
    &\text{because }1+\frac1n\le \frac43,\\
    &\le \frac13 \left(1+\frac1{\beta r}\right) (1+\beta r)^{r} n\\
    &\le \frac{43}{54} \left(\left(\beta+\frac13\right) r\right)^rn
    &\text{because }\beta r\ge \frac{18}{25}.
  \end{align*}
  Therefore if we let $\tau=\max(\frac{1}{3\log 3}, \frac1{\log 2}+
  \frac{\log (\beta+\frac13)}{\log2\log3})$, then
  $2^{\tau r\log r}n\ge ((\beta+\frac13) r)^rn\ge C$.
\end{proof}
When $\beta=10$, $\max(\frac{1}{3\log 3}, \frac1{\log 2}+
\frac{\log (\beta+\frac13)}{\log2\log3})<4.51$ and therefore $\tau=4.51$
satisfies Proposition~\ref{prop:cliquetopminor}.

We can deduce the following theorem
from Proposition~\ref{prop:cliquetopminor}
by using almost identical proofs of Propositions~\ref{prop:p2p3},
\ref{prop:p3p4}, and \ref{prop:p4p5}.
\begin{THM}\label{thm:topminor}
  There is a constant $\tau$ such that
  for every integer $r>2$, if $G$ is a graph with no $K_r$ topological
  minor,
  then
  \begin{align*}
    \cwd(G)&< 2\cdot 2^{\tau r\log r} \rwd(G),\\
    \twd(G)+1&<\frac{3}{4} (r^2+4r-5) 2^{\tau r \log r} \rwd(G).
 \end{align*}
\end{THM}
\begin{proof}
  Let $t=\lceil r/2 \rceil+\binom{\lceil r/2 \rceil}{2}$.
  It is obvious that  $K_{t,t}$ has a topological minor isomorphic to
  $K_r$. So if $G$ is a graph with no $K_r$ topological minor, then
  $G$ has no $K_{t,t}$ subgraph and therefore
  \[
  \twd(G)\le 3(t-1)\cwd(G)
  \]
  by \eqref{eq:gw}.
  From Proposition~\ref{prop:cliquetopminor},
  we can deduce that there is a constant
  $\tau$
  such that $\cwd(G)< 2\cdot 2^{\tau r \log r} \rwd(G)$.
  Thus we deduce the desired inequality, as $t-1\le \frac18 r^2+\frac
  r2-\frac58$.
\end{proof}

\section{Graphs of bounded $\nabla_{1}$}

As mentioned in~\cite{NO2008},  for every $r$ there is a 
function $f$ (resp. $f'$) such that if $G$ is a graph excluding $G$ as 
a minor (resp. topological minor), then $\nabla_{1}(G)\leq f(r)$ (resp. $\nabla_{1}(G)\leq f'(r)$) (see also~\cite{NO2008a}). In that sense, the class of graphs with bounded $\nabla_{1}$ is more general than all the classes we considered in the previous sections. However, the same 
line of arguments allows us to prove that when $\nabla_{1}$ is bounded, then tree-width, rank-width, and clique-width 
are still linearly dependent. For this we first observe the following analogue of Proposition~\ref{prop:p2p3}.

\begin{PROP}
\label{newp2p3}
Let $r\geq 1$. Every $n$-vertex hypergraph $H$ with $\nabla_{1}(I(H))\leq r$ has at most $4^r\cdot n$ hyperedges.
\end{PROP}
\begin{proof} We consider the graph $G$ as in the proof of Proposition~\ref{prop:p2p3} and 
recall that $G$ is a star-minor of $I(H)$. This implies that $\nabla_{1}(G)\leq r$ and thus 
$G$ is $2r$-degenerate. We conclude that $G$ contains at most $4^r\cdot (n-2r+1)\leq 4^{r}\cdot n$ cliques (from~\cite{Wood2007}).
The result follows, as for each hyperedge of $H$, there is a clique in $G$ on the same vertex set.
\end{proof}
It is now easy to produce an analogue of Proposition~\ref{prop:p4p5} by  observing that {\bf i}) $G$ cannot have $K_{2r+1,2r+1}$ as a subgraph (this graph has density more than $r$) and {\bf ii}) if $G'$ is a bipartite subgraph of $G$, then $\nabla_{1}(G')\leq r$. We conclude the following.

\begin{THM}\label{thm:topgrad}
  If $G$ is a graph with at least one edge where $\nabla_{1}(G)\leq r$,
  then
  \begin{align*}
    \cwd(G)&< 2\cdot 4^{r} \rwd(G),\\
    \twd(G)+1&< 12\cdot  r \cdot  4^{r} \rwd(G).
 \end{align*}
\end{THM}

Proposition~\ref{newp2p3} does not hold any more if
we replace $\nabla_1$ with $\nabla_0$: 
The complete graph $K_n$
as a hypergraph
has $\binom{n}{2}$ hyperedges
and yet $I(K_n)$ is $2$-degenerate.

\section{Bounds when excluding $K_{r,r}$ as a subgraph}

%For $r\geq 0$, we denote by $K_{r,r}$ the complete bipartite graph where each part has $r$ vertices.
In this section, we investigate graphs with no $K_{r,r}$ subgraph,
  motivated by the inequality \eqref{eq:gw} of Gurski and Wanke, which
is
\[
\twd(G)\le 2(r-1) \cwd(G)-1.\]
One natural question we might ask is the relation between tree-width
and rank-width for graphs with no $K_{r,r}$ subgraph.
By our approach, 
it is enough to find  an upper bound on the number of hyperedges in a hypergraph
with no $K_{r,r}$ subgraph in its incidence graph.
What are those hypergraphs?
In fact, if $\mathcal F$ is a collection of hyperedges of such a
hypergraph,
then the intersection of $r$ hyperedges can have at most $r-1$ elements.
The problem of finding the maximum possible number of sets with
$k$-wise restricted intersection was studied
more generally by F\"uredi and Sudakov~\cite{FS2004}.
We cite their lemma here. 
\begin{LEM}[F\"uredi and Sudakov {\cite[Lemma 2.1]{FS2004}}]
  Let $k\ge 2$ and $s$ be two positive integers.
  If $\mathcal F$ is a family of subsets of an $n$-element set
  such that $|A_1\cap A_2\cap\cdots \cap A_k|<s$ for all
  $A_1,A_2,\ldots,A_k\in \mathcal F$,
  then
  \[
  |\mathcal F|
  \le \frac{k-2}{s+1}\binom{n}{s}+\sum_{i=0}^s \binom{n}{i}.
  \]
\end{LEM}
(In \cite{FS2004}, $k$ is assumed to be larger than $2$. However, when $k=2$, then the above lemma is implied by a theorem of Frankl and Wilson \cite[Theorem 11]{FW1981}.)
In our case, we let $k=s=r$. Then the above inequality answers P3;
It provides
an upper bound on the
number of hyperedges in a hypergraph whose incidence graph has
no subgraph isomorphic to $K_{r,r}$.
\begin{PROP}[P3]\label{prop:krr}
  Let $H$ be an $n$-vertex hypergraph.
  Let $r\ge 2$.
  If the incidence graph of $H$ has no $K_{r,r}$ subgraph,
  then 
 \[
 \vert E(H) \rvert \le \frac{r-2}{r+1}\binom{n}{r}+\sum_{i=0}^r \binom{n}{i}.
  \]
\end{PROP}
\begin{THM}\label{thm:krr}
  Let $r\ge 2$.
  Let $G$ be an $n$-vertex graph with no subgraph isomorphic to $K_{r,r}$.
  Then
  \begin{align*}
    \cwd(G)&< \frac{2(r-2)}{r+1} \binom{\rwd(G)}{r}+
    2\sum_{i=0}^r \binom{\rwd(G)}{i} ,\\
    \twd(G)+1&<3(r-1)\left(\frac{2(r-2)}{r+1} \binom{\rwd(G)}{r}+
    2\sum_{i=0}^r \binom{\rwd(G)}{i} \right).
 \end{align*}
\end{THM}
\begin{proof}
  From Proposition~\ref{prop:krr},
  $\lambda_G(n)\le \frac{r-2}{r+1} \binom{n}{r}+
  \sum_{i=0}^r \binom{n}{i}$.
\end{proof}

\section{Conclusions}

We observe that Theorem~\ref{thm:minor} has important algorithmic
consequences for approximating rank-width.
By~Feige et al.~\cite{FHL2008}, for every fixed $r$ there exists a polynomial time
constant factor approximating algorithm
computing the tree-width of a graph excluding $K_{r}$ as a minor.
By combining this result with Theorem~\ref{thm:minor}, we deduce that
for every fixed $r$, there is a polynomial time algorithm approximating
within constant factor the  rank-width of a $K_{r}$-minor free graph.

As a side remark, we proved in Proposition~\ref{prop:p2minor}
that every $n$-vertex graph with no $K_r$-minor
has at most $2^{\mu r\log\log r}n$ cliques for a fixed $\mu$.
The previous best upper bound $2^{O(r\sqrt{\log r})}$ was observed by Wood \cite{Wood2007}.
He posed an open problem whether such a graph has at most $c^n$ cliques for a constant $c$.
It will be interesting to resolve this open problem.

%\bibliographystyle{abbrv}
%\bibliography{mybib}
\end{document}